\documentclass[12pt]{article}
\usepackage{amsmath,amssymb,amsfonts,amsthm,enumerate}
\usepackage{eucal} 
\usepackage{xcolor}

\newcommand{\C}{\mathbb{C}}

\newcommand{\Z}{\mathbb{Z}}

\newcommand{\cV}{\mathcal{V}}

\newcommand{\cE}{\mathcal{E}}
\newcommand{\lang}{\left\langle}
\newcommand{\rang}{\right\rangle}

\newcommand{\Om}{\Omega}

\newcommand{\cU}{\mathcal{U}}

\DeclareMathOperator{\Hom}{Hom}

\newtheorem*{lemmaa}{Lemma A}

\newtheorem{thm}{Theorem}

\newtheorem{lemma}{Lemma}
\newtheorem{cor}[lemma]{Corollary}

\begin{document}

\definecolor{chalkboard}{rgb}{0,.15,.15}
%\pagecolor{chalkboard}
%\color{white}

\bibliographystyle{plain}

\title{Localization and a generalization of MacDonald's inner product}
%\title{A generalization of MacDonald's inner product and localization}
\author{Erik Carlsson}

\maketitle
\abstract{
We find a limit formula for a generalization
of MacDonald's inner product in finitely many variables,
using equivariant localization on the Grassmannian
variety, and the main lemma from \cite{Car},
which bounds the torus characters of the higher \c{C}ech
cohomology groups.
We show that the MacDonald inner product conjecture of
type $A$ follows from a special case, and
%using standard symmetric
%function theory in infinitely many variables, as in 
the Pieri rules section of MacDonald's book \cite{Mac},
making this limit suitable replacement for the norm squared of one,
the usual normalizing constant.

}

\section{Introduction}
%\subsection{} 

The MacDonald inner product formula of type $A$
determines the norm squared of the MacDonald
polynomial under the finite variable MacDonald inner product.
It can be obtained up to a constant using
the infinite variable theory, namely the Pieri rules section of MacDonald's
book. However, the normalizing constant, usually taken 
to be the norm squared of one, is difficult to find and is
equivalent to Andrews' $q$-Dyson conjecture. It has been been calculated 
from different angles by a number of authors as 
well as its deep extension to other root systems, see
\cite{BZ,Cher,EtKir,Gar1,GarGon,Kad,Mac1,Mac2,Mim,Stem}.

In this paper, we replace the normalizing constant 
by a limit of norm squares, which we compute for a 
more general inner product.
%for a more general family of inner products.
%The constant, defined as the limit in theorem \ref{mainthm}, 
%looks fairly arbitrary at first
In section \ref{ipsec}, we deduce the inner product formula
from theorem \ref{mainthm}, in the special case of the MacDonald
inner product. This section assumes results about 
MacDonald polynomials in infinitely many variables from
Chapter VI of MacDonald's book, which are in fact the same rules used
to extract the inner product formula from the constant term formula.

%For the purposes of calculating the norm
%squared of the MacDonald polynomial $P_\mu$, we prove that
%it is just as good as knowing the constant term; indeed,
%we deduce the inner product formula using the same results
%taken from the infinite variable theory when the normalizing
%constant is $||1||^2$. On the other hand, theorem \ref{mainthm}
%applies to a much more general family of inner products,
%which will find geometric applications in future papers.

The most interesting feature of this formula 
is that it comes from a geometric argument.
We first identify the inner product with the character
of the space of global sections of a virtual bundle on a
torus equivariant Grassmannian variety.
The main lemma of \cite{Car} shows 
that the higher \c{C}ech cohomology groups do not contribute 
when the bundle is twisted by a high power of the ample
line bundle, so that the character of the space of global sections agrees with
the equivariant Euler characteristic. 
We then use equivariant localization to find a formula for the
Euler characteristic as a sum of factored rational functions.
When the bundle is twisted by a high power of the ample line bundle,
one term dominates the sum, leading to the desired formula.

%using equivariant localization on a Grassmannian variety.
%We then show how inner product formula can be extracted
%from a special case using the infinite variable theory. 
%The key is the main lemma from \cite{Car},
%which bounds the characters of the higher Cech cohomology groups of 
%the integrand.

\section{Plethysm}

Given a smooth complex projective variety $X$, let $K(X)$ denote
its $K$-theory group over the complex numbers. We have the standard
$\lambda$-ring operations defined on bundles by
$\lambda^i([\cV]) = \left[\Lambda^i\cV\right]$, where $[\cV]$ is the
induced class in $K$-theory.
They also act on monomials in every variable in this
paper by
\[\lambda^0(x)=1,\quad \lambda^1(x)=x,\quad 
\lambda^i(x)= 0,\quad i\geq 2.\]
The generating function
\[\lambda(wA) = \sum_{i\geq 0} (-w)^i \lambda^i(A)\]
is the expansion about zero of a rational function of $w$. 
We define $\lambda(A)$ to be the analytic
continuation of this function to $w=1$, if it exists.
For instance, if $x^I$ are formal monomials, we have
\[\lambda\left(\sum_{I} a_I x^I\right) = 
\prod_{I} (1-x^I)^{a_I},\quad
a_I \in \Z.\]
We also have the dualizing operation $[E]^*=[E^*]$, which
acts on monomials by $x^* = x^{-1}$.

Let $\Lambda_n$ denote the graded ring of 
symmetric polynomials in $n$ variables
over $\C$ with inverse limit $\Lambda$,
and let $M$ be a formal 
power series with integer coefficients and constant term $1$ 
in some set of variables.
We define a general MacDonald type inner product on $\Lambda_n$ by
\[(f,g)'_{M} = \lang f(x)g(x^*) \Delta_{M}(x) \rang,\quad
x=\{x_1,...,x_n\}\]
\begin{equation}
\label{sp}
\Delta_M = \lambda\left(M \sum_{i \neq j} x_i x_j^{-1}\right),\quad
\lang F(x)\rang = \frac{1}{n!} [x]_1 F(x).
\end{equation}
Here $[x]_1$ denotes the constant term, which may be extracted
termwise from the power series expansion of $\Delta_M(x)$ in the
indeterminants of $M$, or by assigning those variables values of
norm smaller than one, and defining
\begin{equation}
\label{contour}
[x]_1 = \int_{|x_1|=r} \frac{dx_1}{x_1}\cdots
\int_{|x_1|=r} \frac{dx_n}{x_n},\quad r>1.
\end{equation}
For instance, $M=1$ produces the usual Hall inner product,
while 
\[(f,g)'_{q,t} = (f,g)'_{M},\quad M = \frac{1-t}{1-q}=
(1-t)(1+q+q^2+\cdots),\]
is the finite variable MacDonald inner product
as in \cite{Mac}.

In this paper, we will denote a representation of a one-dimensional complex
torus $\C^*=\{z\}$ and its character in $\Z[z^{\pm 1}]$ by the same letter. 
The $\lambda$-ring notation allows us to define an 
evaluation homomorphism $f \mapsto f(A)$ for any virtual character $A$
by its action on generators
\[\Lambda \rightarrow \Z[z^{\pm 1}],\quad
e_i \mapsto \lambda^i(A),\quad A \in \Z[z^{\pm 1}],\]
where $e_i$ are the elementary symmetric functions.
Let us set
\[\Om(A) = \lambda(Ax)^{-1} = \sum_{\mu} u_\mu(A) v_\mu\]
for any dual bases $u_\mu,v_\mu$ of $\Lambda$. In particular,
$\Om = \Om(1)$ is the sum of the complete symmetric polynomials.
We also extend the dimension map to any virtual character
$A \in \Z[z^{\pm 1}]$ by
\[\dim_a(A) = [z^a] A,\quad a \in \Z,\quad \dim(A) = \sum_a \dim_a(A),\]
where $[z^a]$ denotes the coefficient of $z^a$.

\section{The main theorem}

Given a one dimensional complex torus $T = \{z\}$
and a finite dimensional representation $Z$, let
$X$ denote the Grassmannian of subspaces of $Z$
of codimension $n$. We label a point in this variety
by its $n$-dimensional quotient space $U=Z/V$, rather than 
the space itself. Let $\cU$ denote the $n$-dimensional 
tautological quotient bundle on $X$ whose fiber over 
$U \in X$ is $U$ itself. Let $K_T(X)$ denote the equivariant
$K$-theory group with respect to the induced action of $T$ on $X$.

The facts we will need about this variety are summarized below.

\begin{enumerate}
\item We have an algebra homomorphism defined on generators by
\[\Lambda_n \otimes \Lambda_n \rightarrow K_T(X),\quad
e_i \otimes e_j \mapsto \Lambda^i(\cU) \otimes \Lambda^j(\cU)^*,\]
where $e_i$ is the $i$th elementary symmetric function.
We denote the image of $f\otimes g$ by $f(\cU)g(\cU^*)$.

\item We have a linear map
$\chi^0 : \Lambda_n \otimes \Lambda_n \rightarrow \C[t^{\pm 1}]$ 
defined on Schur positive elements by
\[\chi_{X}^0(f(\cU)g(\cU^*)) = H^0_{X}(f(\cU)g(\cU^*)) \in \Z[T].\]
%
%whenever $f,g$ are Schur positive, so that $f(\cU)g(\cU^*)$ is
%an honest bundle, and which extends by linearity to any
%$f,g$. 
The following formula follows from Borel-Weil-Bott
and also \cite{Ed}.
% and can
%be seen, for instance from the Shaun Martin or Jeffrey-Kirwan 
%formula \cite{Ma,JK}:
%
\begin{equation}
\label{bwb}
\chi^0_{X}(f(\cU)g(\cU^*)) = 
\lang f(x)g(x^*)\lambda(x^*Z)^{-1}\Delta(x)\rang.
\end{equation}
The fractional term refers to its expansion
about $x=\infty$.
%from the expansion about $x_i=\infty$, due to the fractional
%term. 
%See \cite{C} for more on these contours. 
%
\item We also have the equivariant Euler characteristic
\[\chi : K_T(X) \rightarrow \C[t^{\pm 1}],\quad
\chi([\cE]) = \sum_{i\geq 0} (-1)^i H^i_X(\cE).\]
%
%Atiyah-Bott-Leftschetts localization gives a formula for the Euler
%characteristic in terms of fixed-point data. 
If the torus fixed points $U \in X^T$ are isolated, then they are
indexed by the $n$ element subsets of the weight set. 
The character of the cotangent space to such a point is 
%
%\begin{equation}
%\label{cotan}
\[T^*_U X = \Hom(V,U)^* = U^*V.\]
%\end{equation}
%
The $K$-theoretic localization formula in this context says that
\begin{equation}
\label{loc}
\chi(f(\cU)g(\cU^*)) =
\sum_{U \in X^T} f(U)g(U^*) \lambda(U^*V)^{-1}.
\end{equation}
See \cite{Ginz} for a reference.
\end{enumerate}

Now let $A,B,C$ be power series in $\Z((z))$, and define
\[\cE_{\cU} = A \cU+B\cU^*+C\cU\cU^* \in K_T(X) \otimes \C[[z]].\]
The following is the main lemma from \cite{Car}.

\begin{lemmaa}
Suppose the following conditions hold:
\begin{enumerate}[a)]
\item $A \in \Z_{\geq 0}[z^{\pm 1}]$.
\item $B \in \Z_{\geq 0}((z))$, and
$\dim_a(B) \leq p(a)$ for some polynomial $p(a)$.
\item $C \in (z)\Z[z]$ is a polynomial with constant term zero.
\item For any weights $a,b\in \Z$ with
$\dim_{a}(Z)>0$, $\dim_b(C) <0$,
we have 
\[\dim_{a+b}(B) \geq \dim_a(Z)-\dim_b(C)-1.\]
\end{enumerate}
%
%\end{enumerate}
Then we have that
\begin{equation}
\label{lemmaaeq}
\sum_{i\geq 0}(-w)^i\left(\chi-\chi^0\right) 
\left(\det(\cU)^m \lambda^i\left(\cE_{\cU}\right)\right) \in \C((z))[[w]]
\end{equation}
is the $w$-expansion of a meromorphic function $c_{m}(z,w)$, 
and the leading exponent of the expansion of 
$c_{m}(z,w)$ in the $z$ direction is bounded below by
$mk+c$, where $k$ is the largest torus weight of $Z$, 
and $c$ is some constant.
\end{lemmaa}
%
%\begin{lemmaa1}

\emph{Remark.} This does not mean that the coefficient 
$[w^i]c_m(z,w)$ has high leading degree. 
In fact, Serre duality implies
that these terms should contribute large negative powers of $z$,
and would therefore dominate the power series.

Since $\Lambda$ is spanned by expressions of the form
$\lambda(Ax)$ for a general $A$, we may insert a factor
of $f(\cU)$ into \eqref{lemmaaeq}, for arbitrary $f \in \Lambda$.
Then apply \eqref{bwb} and \eqref{loc} to get
\[c_{m}(z,w)=\lang e_n^m f(x) 
\lambda\left(w\cE_x\right)\lambda(Zx^*)^{-1}\Delta(x)\rang-\]
\begin{equation}
\label{lemmaa1}
\sum_{U \in X^T} e_n^mf(U)\lambda
\left(w\cE_U\right)
\lambda(T^*_U X)^{-1}.
\end{equation}

\begin{thm}
\label{mainthm}
Let $M \in \Z[z]$ be a polynomial with constant term $1$,
and let
\[E = MU_n U_n^*-U_n^*,\quad U_n = 1+\cdots+z^{n-1}.\]
Suppose that $\dim_a(M)\geq -1$ for all $a \in \Z$, and 
$\dim_0(E)=0$, so that the right hand side below is defined and
nonzero. Then
\begin{equation}
\label{thmeq}
\lim_{m \rightarrow \infty} z^{-m{n \choose 2}}(fe_n^m,\Omega)'_{M}=
\lambda(1-M)^nf(U_n)
\lambda\left(E\right).
\end{equation}
%
%where
%
%\[M = (1-t)(1+\cdots+q^{N-1}),\quad U_{\rho} = 1+\cdots+z^{n-1}.\]
%$U_\rho = 1+\cdots+z^{n-1}$.
\end{thm}

\begin{proof}

Let 
\[A = 0,\quad B = z+z^2+\cdots,\quad \]
\[C = M-1,\quad Z = 1+\cdots+z^k,\]
which satisfy the conditions of the lemma.
For large $m$, the
sum in \eqref{lemmaa1} is dominated by $U=U_n$. 
By the lemma we have
\[\lim_{m\rightarrow \infty} z^{-m{n \choose 2}} \lang e_n^mf(x)
\lambda\left(w\cE_x\right)\lambda(Zx^*)^{-1}\Delta(x)\rang=\]
\[f(U_n)\lambda\left(w \cE_{U_n}\right)\lambda(T^*_{U_n} X)^{-1},\quad
k > {n \choose 2}.\]
To complete the proof, interpret the above as an equality of rational
functions in $z,w$, and set $w=1$. Then take the limit as $k$ tends to infinity,
and notice that
\[\Delta(x)\lambda(Cxx^*)=\lambda(C)^n\Delta_{M}(x).\]

\end{proof}

\emph{Remark.} We must take the limit over $m$ before setting $w=1$,
because there are examples of fixed points $U \in X^T$
such that $\lambda(w\cE_U)$ has a pole at $w=1$.
This is the reason we use the main lemma from \cite{Car}
rather than the main theorem.

\section{The inner product formula}

\label{ipsec}

\begin{cor}
\label{thecor}
We have
\[(P_\mu,P_\mu)'_{q,t} = 
\lambda\left(n\frac{t-q}{1-q}+\frac{q-t^n}{1-q}t^{-\rho}\right)\]
\begin{equation}
\label{nsa}
P_{\mu}(t^\rho)\lim_{m\rightarrow \infty}a_{\mu+m^n}(q,t)^{-1},
\end{equation}
where $a_\mu(q,t)$ is the coefficient of $P_\mu$ in the expansion
of $\Omega$ in the Macdonald basis, and
\[\mu+m^n = [\mu_1+m,...,\mu_n+m],\quad 
u^{\rho} = 1+\cdots+u^{n-1}.\]
\end{cor}

\begin{proof}
We first claim that equation \eqref{thmeq} holds at
\[M = (1-t)(1+\cdots+q^{N-1}),\quad z=t.\]
It is not hard to see that the condition of the theorem holds
at $q=t^k$ for $k$ larger than $n$, whence the claim holds
for that specialization.

To establish the claim, it suffices to prove that both 
sides are rational functions of $q,t$, because rational functions
that agree at $q=t^k$ for infinitely many values of $k$ must be equal.
Let $F_m(q,t)$ denote the quantity inside the limit in the
left hand side.
Using the contour description of the inner product \eqref{contour},
and Cauchy's residue formula,
% to the contours in $x_i$ in sequence, 
we see that there exists $G(q,t,x,y)$ such that
\[F_m(q,t) = G(q,t,q^m,t^m),\quad G(q,t,x,y) \in \C(q,t)[x,y],\]
when $m$ is large enough that there are no singularities at
$x_i=0$. The limit is obtained by setting $x=y=0$,
proving that the left side is a rational function of $q,t$.
The right hand side is obviously also such a function.
%proving the claim.

Now let $f=P_\mu$ and take the limit over $N$.
The lemma follows because the MacDonald
polynomials are orthogonal, and
%
%\[
\[e^m_n P_\mu(x) = P_{\mu+m^n}(x).\]
%,\quad 
%\mu+m^n = [\mu_1+m,...,\mu_n+m],\]
%
%when the number of variables is $n$.

\end{proof}

We will now recover
%with corollary \ref{thecor} recovers 
the inner product formula. 
Given a Young diagram $\mu$, and a square
$s\in \mu$, let $a(s),l(s),a'(s),l'(s)$ denote the 
arm, leg, coarm, and coleg lengths. Let
\[c_\mu(q,t) = \prod_{s\in \mu} (1-q^{a(s)}t^{l(s)+1}),\quad
c'_\mu(q,t) = \prod_{s\in \mu} (1-q^{a(s)+1}t^{l(s)}),\]
so that
\begin{equation}
\label{nsi}
(P_\mu,P_\mu)_{q,t} = \frac{c'_\mu(q,t)}{c_\mu(q,t)},
\end{equation}
as in chapter VI, section 6 of MacDonald's book \cite{Mac}.

\begin{lemma}
The coefficients from corollary \ref{thecor} are given by
\[a_\mu(s) = c'_\mu(q,t)^{-1}\prod_{s\in \mu}(t^{l'(s)}-q^{a'(s)+1}).\]

\end{lemma}

\begin{proof}
We have
\[(f,\Omega)_{q,t} = 
\left(f,\Omega\left(\frac{1-q}{1-t}\right)\right) = 
\varepsilon_{q,t}(f),\]
where $\varepsilon_{u,t}$ is the homomorphism 
defined on the power sum generators of $\Lambda$ by
\[\varepsilon_{u,t}(p_j) = \frac{1-u^j}{1-t^j}.\]
Equation (6.17), chapter VI, of \cite{Mac} says that
\begin{equation}
\label{eps}
\varepsilon_{u,t} (P_\mu) = c_\mu(q,t)^{-1}\prod_{s \in \mu}
(t^{l'(s)}-q^{a'(s)}u).
\end{equation}
The answer follows by setting $f=P_\mu$, and using \eqref{nsi}.

\end{proof}

\begin{cor}
We have the inner product formula
\begin{equation}
\label{ns}
(P_\mu,P_\mu)'_{q,t} = 
\lambda\left((t-q)\frac{1-t}{1-q} 
\sum_{1\leq i<j \leq n} q^{\mu_i-\mu_j}t^{j-i-1}\right).
\end{equation}
\end{cor}

\begin{proof}

Let us take the limit in equation \eqref{nsa}.
We will suppress the limit symbols, and instead
suppose that $\nu=\mu+m^n$ for very large $m$. 
Using equation \eqref{eps} to evaluate $P_\mu(t^\rho)$, we have
\[(P_\mu,P_\mu)'_{q,t} = 
\lambda\left(n\frac{t-q}{1-q}-t^{-\rho}+\frac{1-t}{1-q}t^{\rho}t^{-\rho}\right)\]
\[\frac{c'_{\nu}(q,t)}
{c_{\nu}(q,t)}
\prod_{s\in \nu} \frac{t^{l'(s)}-t^nq^{a'(s)}}{t^{l'(s)}-q^{a'(s)+1}}=
\lambda(A),\]
where
\[A =n\frac{t-q}{1-q}-t^{-\rho}+\frac{1-t^n}{1-q}t^{-\rho}+\]
\[\sum_{s\in\nu} t^{-l'(s)}q^{a'(s)}(t^n-q)+\]
\[\sum_{s\in\nu}q^{a(s)+1}t^{l(s)}-q^{a(s)}t^{l(s)+1}=\]
\[n\frac{t-q}{1-q}+(q-t)\sum_{s\in \nu} q^{a(s)}t^{l(s)}=\]
\[(t-q)\left(\frac{n}{1-q}-\sum_{s\in \nu} q^{a(s)}t^{l(s)}\right)=\]
\[(t-q)(1-t)\sum_{k\geq 0} c_{k}(t) q^k,\]
\[c_k(t) = \sum_{s\in\nu,\,a(s)=k} (1+\cdots+t^{l(s)-1}).\]
Since $m$ is large, there is a unique box $s$ in each row $i$
with arm length $k$. It is straightforward to 
match terms under this correspondence
with equation \eqref{ns}, completing the proof.

\end{proof}

\bibliography{ctbib}{}
\bibliographystyle{plain}

\end{document}